\documentclass[11pt]{article}

\usepackage[T1]{fontenc}
\usepackage[a4paper,margin=29mm]{geometry}
\usepackage{amsmath,amssymb,amsthm}
\usepackage[hidelinks]{hyperref}

\allowdisplaybreaks

\newtheorem{theorem}{Theorem}[section]
\newtheorem{lemma}[theorem]{Lemma}
\newtheorem{proposition}[theorem]{Proposition}

\theoremstyle{definition}

\theoremstyle{remark}

\newcommand{\Z}{\mathbb Z}
\newcommand{\F}{\mathbb F}
\newcommand{\ICG}{\operatorname{ICG}}
\newcommand{\Spec}{\operatorname{Spec}}
\newcommand{\Div}{\operatorname{Div}}
\newcommand{\supp}{\operatorname{supp}}
\newcommand{\M}{\mathsf M}
\newcommand{\atom}[1]{\left[#1\right]}
\newcommand{\eps}{\varepsilon}
\newcommand{\Q}{\mathcal Q}

\title{So's Conjecture for Integral Circulant Graphs\\
of Order \(p^a q\)}
\author{Jianwei Jiang \qquad Chunhua Yang\thanks{Corresponding author.}\\[0.5ex]
\small School of Mathematics and Statistics, Weifang University, Weifang, Shandong, China\\
\small \texttt{jianweijiangwf@sina.com} \qquad \texttt{chunhuayangwf@sina.com}}
\date{}

\begin{document}

\maketitle

\begin{abstract}
So conjectured that, for fixed \(n\), the ordinary adjacency spectrum of
\(\ICG(n,\mathcal D)\) determines its divisor set \(\mathcal D\). We prove this
for \(n=p^aq\), where \(p<q\) are primes and \(a\geq1\). To handle coincident
eigenvalues from distinct gcd classes, we encode the ordinary spectrum by a
spectral counting measure. For connected \(G=\ICG(p^aq,\mathcal D)\), an exact
identity recovers \(T=\mathcal D\cap\{1,q\}\) and the counting measure of
\(H=\ICG(p^{a-1}q,\mathcal E)\), where
\(\mathcal D=p\mathcal E\mathbin{\dot\cup}T\). Strong induction and the component
decomposition then recover \(\mathcal D\), including when \(p=2\) and when
\(G\) is disconnected.
\end{abstract}

\noindent\textbf{Keywords.}
Spectral graph theory; isospectrality; Ramanujan sums; spectral counting measure;
gcd classes.

\medskip
\noindent\textbf{2020 Mathematics Subject Classification.}
05C50, 05C25, 11L03.

\section{Introduction}

For a positive integer \(n\), So's characterization \cite{So2006}
identifies integral circulant graphs with subsets of the proper divisors
of \(n\). More precisely, if
\(\mathcal D\subseteq\Div^*(n)\), where \(\Div^*(n)\) denotes the set of
proper divisors of \(n\), then the integral circulant graph
\(\ICG(n,\mathcal D)\) has vertex set \(\Z/n\Z\), and two vertices are
adjacent precisely when the greatest common divisor of their difference
with \(n\) belongs to \(\mathcal D\). We refer to \(\mathcal D\) as the
divisor set of the graph. So conjectured that, for fixed \(n\), the
ordinary adjacency spectrum as a multiset (hereafter, the ordinary
spectrum) uniquely determines this divisor set.

The distinction between the ordinary spectrum and the spectral vector is essential. By the standard Ramanujan-sum formula for integral circulant graphs \cite{So2006}, the spectral vector \((\lambda_k)_{k\in\Z/n\Z}\) is constant on gcd classes, whereas the ordinary spectrum records only the numerical eigenvalues and their multiplicities. Distinct spectral-vector entries may therefore
coincide numerically, in which case their multiplicities are merged. Thus
the ordinary spectrum does not retain the frequency labels needed to
reconstruct the spectral vector directly. To retain the multiplicity
information through such collisions, we encode the ordinary spectrum by
its spectral counting measure. The splitting and congruence identities used
below are then identities of counting measures, so they remain valid even
when different spectral-vector entries coincide. This keeps the recovery
argument entirely at the ordinary-spectrum level, without assuming access
to the lost gcd-class labels.

Along the family \(p^aq\), earlier ordinary-spectrum results cover
\(pq\) \cite{So2006}, \(p^2q\) \cite{Cusanza2006,MoniusSo2023}, and
\(p^3q\) \cite[Theorem~1.3(c)]{LiLiu2024}. Li and Liu also proved, as
part of a more general separated-prime theorem, the odd family \(p^aq\)
under the additional condition \(p^a<q\)
\cite[Theorem~1.3(a)]{LiLiu2024}. The complementary family \(pq^k\), with
\(p<q\) and \(k\geq1\), was treated in
\cite{Cusanza2006,MoniusSo2023}. The case \(p^2q^2\) is covered by Li and Liu
\cite[Theorem~1.3(d)]{LiLiu2024} and by Zhang
\cite[Theorem~7.4]{Zhang2023}, who also obtained further conditional results
for broader prime-factorization patterns. See also the recent
discussion in \cite{AltinisikAydin2024}.

Related work also concerns the labeled spectral vector and the isomorphism
problem. Sander and
Sander obtained a spectral-vector result for multiplicative divisor sets
\cite[Theorem~1.2]{SanderSander2015}; related structural formulas appear in
\cite{Sander2018}; and Schlage-Puchta established injectivity of the
spectral-vector map and derived further sufficient conditions for So's
conjecture \cite[Corollary~3 and Proposition~4]{SchlagePuchta2021}. Klin and
Kov\'acs settled the corresponding isomorphism question for rational
circulant graphs \cite[Corollaries~11.2--11.3]{KlinKovacs2012}. Thus the issue addressed here is uniqueness from the ordinary spectrum. The result below
removes both the bound on the exponent and the prime-separation condition
and treats the full family \(p^aq\), including \(p=2\).

For a graph \(G\), write \(\Spec(G)\) for its adjacency spectrum as a
multiset.

\begin{theorem}\label{thm:main}
Let \(p<q\) be primes and let \(a\geq1\). If
\(\mathcal D_1,\mathcal D_2\subseteq\Div^*(p^aq)\) and
\[
\Spec\!\left(\ICG(p^aq,\mathcal D_1)\right)
=
\Spec\!\left(\ICG(p^aq,\mathcal D_2)\right),
\]
then \(\mathcal D_1=\mathcal D_2\).
\end{theorem}

The central step is Theorem~\ref{thm:w1-recovery}. We separate the part of
the divisor set not divisible by \(p\) by writing
\[
\mathcal D=p\mathcal E\mathbin{\dot\cup}T,
\qquad
T\subseteq\{1,q\},
\qquad
\mathcal E\subseteq\Div^*(p^{a-1}q),
\]
where \(p\mathcal E:=\{pe:e\in\mathcal E\}\), and put
\(H:=\ICG(p^{a-1}q,\mathcal E)\). For connected
\(G=\ICG(p^aq,\mathcal D)\), the ordinary spectrum determines both \(T\)
and the exact spectral counting measure of \(H\), without requiring \(H\)
to be connected or its divisor set to be nonempty. The proof starts from an exact identity relating the counting measures of
\(G\) and \(H\). Depending on \(v_p(q-1)\), we either recover \(\M_H\) directly
from congruences modulo \(p\), invoke the prime-separation result of Li and Liu
to determine the whole divisor set, or repeatedly divide an exact measure
identity by \(p\). The remaining cases are reduced, by taking prefix sums, to
comparisons of finite integer intervals. When \(p=2\), the recursion may
continue to the graph of order \(q\). A strong induction on \(a\), together with the component decomposition
for disconnected graphs, then recovers the full divisor set.

\section{Preliminaries and spectral counting measures}

Write \(\Div^*(n)=\{d\in\Z_{>0}:d\mid n,\ d<n\}\).

For integers \(m\geq1\) and \(k\), the Ramanujan sum is defined by
\[
c_m(k)=
\sum_{\substack{1\leq x\leq m\\(x,m)=1}}
\exp(2\pi \mathrm{i}xk/m).
\]
We write \(\varphi\) for Euler's totient function and, for a prime
\(\ell\), write \(\F_\ell=\Z/\ell\Z\). For standard properties of
Ramanujan sums used below, see, for example, Apostol
\cite[Secs.~8.3--8.4]{Apostol1976}. If \(\ell\) is prime, then for
\(r\geq1\) and \(v\geq0\),
\begin{equation}\label{eq:gen-prime-power-ramanujan}
c_{\ell^r}(\ell^v)=
\begin{cases}
0,&0\leq v\leq r-2,\\
-\ell^{r-1},&v=r-1,\\
\ell^{r-1}(\ell-1),&v\geq r,
\end{cases}
\qquad c_1(k)=1.
\end{equation}

For primes \(p<q\) and integers \(a,b\geq1\), let \(n=p^aq^b\) and
let \(\mathcal D\subseteq\Div^*(n)\). Write \(v_p\) and \(v_q\) for
the corresponding \(p\)-adic and \(q\)-adic valuations. For a frequency
\(k\in\Z/n\Z\), put \(u=v_p(\gcd(k,n))\) and
\(v=v_q(\gcd(k,n))\).
Summing So's formula \cite[Theorem~5.1]{So2006} over
\(d\in\mathcal D\) gives the adjacency eigenvalue of \(\ICG(n,\mathcal D)\)
indexed by \(k\in\Z/n\Z\):
\begin{equation*}
\lambda_k(\mathcal D)
=
\sum_{d\in\mathcal D}c_{n/d}(k)
=
\sum_{p^iq^j\in\mathcal D}
c_{p^{a-i}}(p^u)c_{q^{b-j}}(q^v).
\end{equation*}
The second equality follows from the multiplicativity of Ramanujan sums
for coprime moduli together with \eqref{eq:gen-prime-power-ramanujan},
which shows that each prime-power factor depends only on the corresponding
valuation. In particular, \(\lambda_k(\mathcal D)\) depends only on
\(\gcd(k,n)\) and is integral.

Since the ordinary spectrum does not retain the gcd labels, numerically
equal spectral-vector entries contribute only their combined multiplicity.
We therefore encode the ordinary spectrum by its spectral counting measure.
Let \(\mathcal A=\Z[\Z]\) be the free abelian group with basis
\(\{\atom z:z\in\Z\}\). For an integral circulant graph \(G\), define
\begin{equation*}
\M_G=\sum_{z\in\Z}m_G(z)\atom z,
\end{equation*}
where \(m_G(z)\) is the multiplicity of \(z\) in the ordinary spectrum of
\(G\). Congruences modulo a prime are taken coefficientwise; the integers \(z\)
indexing the basis elements \(\atom z\) are unchanged. Division is also
coefficientwise and is used only after exact divisibility in \(\mathcal A\)
has been established.

\begin{lemma}\label{lem:gen-components}
Let \(\varnothing\neq\mathcal D\subseteq\Div^*(n)\), put
\(g=\gcd(\mathcal D)\), and define
\(\mathcal D/g:=\{d/g:d\in\mathcal D\}\). Then
\(\ICG(n,\mathcal D)\) is the disjoint union of \(g\) isomorphic
connected components, each isomorphic to \(\ICG(n/g,\mathcal D/g)\).
Moreover, the ordinary spectrum determines
\(g\) and hence the ordinary spectrum of one component.
\end{lemma}

\begin{proof}
The standard component criterion gives \(g\) connected components; see
\cite[Lemma~1.1]{MoniusSo2023}. Thus each component is naturally isomorphic to \(\ICG(n/g,\mathcal D/g)\). Since each component is connected and regular,
its degree is the largest adjacency eigenvalue, and this eigenvalue is
simple. Hence the largest eigenvalue of \(\ICG(n,\mathcal D)\) is this common
degree and occurs once in each of the \(g\) components, so its multiplicity
is exactly \(g\). Dividing every eigenvalue multiplicity by \(g\) then
recovers the ordinary spectrum of \(\ICG(n/g,\mathcal D/g)\).
\end{proof}

\begin{lemma}\label{lem:gen-prime-power}
Let \(\ell\) be a prime and \(r\geq1\). For every
\(\mathcal D\subseteq\Div^*(\ell^r)\), the ordinary spectrum of
\(\ICG(\ell^r,\mathcal D)\) uniquely determines \(\mathcal D\).
\end{lemma}

\begin{proof}
This follows from \cite[Corollary~2.2]{SanderSander2015}; in fact, for
prime-power order the largest eigenvalue already determines the divisor set.
\end{proof}

\section{General \texorpdfstring{\(p\)}{p}-layer spectral decomposition}

Fix \(n=p^a q^b\), where \(p<q\) are primes and \(a,b\geq1\). For
\(\mathcal D\subseteq\Div^*(n)\), put \(G=\ICG(n,\mathcal D)\). Every such
divisor set has a unique split
\begin{equation*}
 \mathcal D=p\mathcal E\mathbin{\dot\cup}T,
 \qquad T\subseteq\{1,q,\ldots,q^b\},
\end{equation*}
where \(\mathcal E\subseteq\Div^*(p^{a-1}q^b)\). Put
\(H=\ICG(p^{a-1}q^b,\mathcal E)\), \(h=p^{a-1}\), and \(P=p^a\).
Retaining the valuation notation from Section~2, for
\(0\le u\le a\) and \(0\le v\le b\), let \(\lambda_{u,v}\) denote the
common spectral-vector entry of \(G\) on the gcd class corresponding to the
valuation pair \((u,v)\). For \(0\le u\le a-1\) and \(0\le v\le b\), let
\(\mu_{u,v}\) denote the corresponding spectral-vector entry of \(H\).
Thus the two highest \(p\)-layers of \(G\), namely \(u=a-1\) and \(u=a\),
both correspond to the top \(p\)-layer \(u=a-1\) of \(H\). We refer to a gcd class with valuation pair \((u,v)\) simply as the gcd
class \((u,v)\); in \(G\) it has
multiplicity
\(\varphi(p^{a-u})\varphi(q^{b-v})\).
To encode the \(q\)-part of the contribution of \(T\), for
\(0\le j\le b\) let \(\eps_j\in\{0,1\}\) be the indicator of the
condition \(q^j\in T\), and define \(\tau_v=\sum_{j=0}^b \eps_jc_{q^{b-j}}(q^v)\) for
\(0\le v\le b\). Set \(\gamma_v=\mu_{a-1,v}\).

The next two results are stated for arbitrary \(b\); in
Section~\ref{sec:w1} we apply them with \(b=1\).

\begin{proposition}\label{prop:gen-spectral-vector-split}
For \(0\leq u\leq a-2\), \(\lambda_{u,v}=\mu_{u,v}\). At the two highest \(p\)-layers,
\begin{equation*}
 y_v:=\lambda_{a-1,v}=\gamma_v-h\tau_v,
 \qquad
 x_v:=\lambda_{a,v}=\gamma_v+(p-1)h\tau_v.
\end{equation*}
In particular,
\begin{equation}\label{eq:gen-top-difference}
 x_v-y_v=P\tau_v,
 \qquad \gamma_v=y_v+h\tau_v.
\end{equation}
\end{proposition}

\begin{proof}
We first consider the contribution of \(p\mathcal E\) to the spectral-vector entry \(\lambda_{u,v}\). For \(pe\in p\mathcal E\),
\(p^aq^b/(pe)=p^{a-1}q^b/e\), so the corresponding Ramanujan modulus in
\(G\) is exactly the modulus attached to \(e\) in
\(H=\ICG(p^{a-1}q^b,\mathcal E)\). Hence, for \(0\leq u\leq a-2\), lowering the \(p\)-exponent by one does not change the gcd-class label \((u,v)\), so the contribution of \(p\mathcal E\) to \(\lambda_{u,v}\) is \(\mu_{u,v}\). The two highest \(p\)-layers of \(G\), \(u=a-1\) and \(u=a\), both correspond to the top \(p\)-layer of \(H\), so in both cases the contribution of \(p\mathcal E\) to \(\lambda_{u,v}\) is \(\mu_{a-1,v}=\gamma_v\).

The contribution of \(T\) to the spectral-vector entry \(\lambda_{u,v}\) is \(c_{p^a}(p^u)\tau_v\). By \eqref{eq:gen-prime-power-ramanujan},
\[
c_{p^a}(p^u)=
\begin{cases}
0,&0\leq u\leq a-2,\\
-h,&u=a-1,\\
(p-1)h,&u=a.
\end{cases}
\]
Combining the contributions of \(p\mathcal E\) and \(T\) to \(\lambda_{u,v}\) gives
\[
\lambda_{u,v}=\mu_{u,v}\quad(0\leq u\leq a-2),\qquad
y_v=\gamma_v-h\tau_v,\qquad
x_v=\gamma_v+(p-1)h\tau_v.
\]
The remaining identities follow immediately from these formulas.
\end{proof}

Put
\begin{equation*}
 \omega_v=\varphi(q^{b-v})=
 \begin{cases}
 q^{b-v-1}(q-1),&v<b,\\
 1,&v=b,
 \end{cases}
\end{equation*}
and, for integers \(y,\tau\), define
\begin{equation}\label{eq:gen-Q}
 \Q_\tau(y)=(p-1)\atom y+\atom{y+P\tau}
              -p\atom{y+h\tau}.
\end{equation}

\begin{proposition}
\label{prop:gen-measure-split}
In the free abelian group \(\mathcal A\),
\begin{equation}\label{eq:gen-measure-split}
 \boxed{
 \M_G=p\M_H+\sum_{v=0}^b \omega_v\Q_{\tau_v}(y_v).}
\end{equation}
\end{proposition}
\begin{proof}
For \(0\leq u\leq a-2\), Proposition~\ref{prop:gen-spectral-vector-split}
gives \(\lambda_{u,v}=\mu_{u,v}\).  The corresponding gcd-class multiplicities in \(G\) and \(H\) are
\(\varphi(p^{a-u})\varphi(q^{b-v})\) and
\(\varphi(p^{a-1-u})\varphi(q^{b-v})\), respectively.  Since
\(\varphi(p^{a-u})=p\,\varphi(p^{a-1-u})\), the gcd classes with
\(0\leq u\leq a-2\) contribute exactly the corresponding part of
\(p\M_H\). It remains to replace the contribution of the top
\(p\)-layer of \(H\). For fixed \(v\), the \(q\)-factor in the
gcd-class multiplicity is \(\omega_v\), so \(p\M_H\) contributes
\(p\omega_v\atom{\gamma_v}\), whereas the two top \(p\)-layers of \(G\)
contribute
\[
(p-1)\omega_v\atom{y_v}+\omega_v\atom{x_v}.
\]
Hence the correction for this fixed \(v\) is
\[
\omega_v\bigl((p-1)\atom{y_v}+\atom{x_v}-p\atom{\gamma_v}\bigr)
=
\omega_v\Q_{\tau_v}(y_v),
\]
where \(x_v=y_v+P\tau_v\) and \(\gamma_v=y_v+h\tau_v\).
Summing over \(v\) gives \eqref{eq:gen-measure-split}.

If several displayed integers are equal, their coefficients are simply added.
\end{proof}

\section{Recovery for \texorpdfstring{\(p^a q\)}{p to the a q}}\label{sec:w1}

We now specialize to \(b=1\). Thus \(T\subseteq\{1,q\}\) and
\((\omega_0,\omega_1)=(q-1,1)\). With \(\eps_0,\eps_1\) as defined in Section~3, put
\(c:=\tau_0=\eps_1-\eps_0\) and
\(d:=\tau_1=\eps_1+(q-1)\eps_0\). The four possible values of \((T,c,d)\) are
\begin{equation}\label{eq:w1-state-table}
\begin{array}{c|rrrr}
T&\varnothing&\{q\}&\{1\}&\{1,q\}\\ \hline
c&0&1&-1&0\\
d&0&1&q-1&q.
\end{array}
\end{equation}

Specializing Proposition~\ref{prop:gen-measure-split} to \(b=1\) gives
\begin{equation}\label{eq:w1-special-split}
\M_G=p\M_H+(q-1)\Q_c(y_0)+\Q_d(y_1).
\end{equation}
For the recovery arguments below, assume that the original graph \(G\) is
connected. Neither \(H\) nor the intermediate graphs need be connected.
Lemma~\ref{lem:w1-large-valuation} also holds without this assumption on \(G\).
The frequency \(k=0\), corresponding to the trivial character, lies in the
gcd class \((a,1)\). Since every
Cayley graph is regular, the corresponding eigenvalue is its degree
\(\Delta\), so \(\lambda_{a,1}=x_1=\Delta\). Since \(G\) is connected, \(\Delta\)
is also the largest adjacency eigenvalue and has multiplicity one. By
\eqref{eq:gen-top-difference},
\begin{equation}\label{eq:w1-known-top-relations}
y_1=\Delta-Pd,\qquad
\gamma_1=y_1+hd=\Delta-(p-1)hd,\qquad
x_0-y_0=Pc.
\end{equation}

Our first task is to recover \(T\) from the ordinary spectrum. Once \(T\)
is known, \(c,d\), and hence \(y_1\), follow from
\eqref{eq:w1-state-table}--\eqref{eq:w1-known-top-relations}. It remains to
recover \(y_0\), or equivalently the pair \((x_0,y_0)\).

\subsection{Recovering \texorpdfstring{\(T\)}{T} and the case \texorpdfstring{\(p\nmid(q-1)\)}{p not dividing q-1}}

\begin{lemma}\label{lem:w1-top-state}
If \(G=\ICG(p^a q,\mathcal D)\) is connected, its ordinary spectrum
uniquely determines \(T=\mathcal D\cap\{1,q\}\).
\end{lemma}

\begin{proof}
Since \(G\) is connected and \(p^aq>1\), its divisor set \(\mathcal D\) is
nonempty. If \(T=\varnothing\), then every divisor in \(\mathcal D\) is
divisible by \(p\), so \(\gcd(\mathcal D)\ge p\), contradicting
Lemma~\ref{lem:gen-components}. Thus \(T\ne\varnothing\), and the ordinary
spectrum determines \(\Delta=x_1\).

Let \(\alpha\in\F_p\) be the residue class of \(q-1\), and let \(\rho\) be
the coefficientwise reduction of \(\M_G-\atom\Delta\) modulo \(p\). Since
\(\Q_\tau(y)\equiv\atom{y+P\tau}-\atom y\pmod p\), reducing
\eqref{eq:w1-special-split} modulo \(p\) gives
\begin{equation}\label{eq:w1-top-signature}
\rho=-\atom{y_1}+\alpha(\atom{x_0}-\atom{y_0}),
\qquad
y_1=\Delta-Pd,\quad x_0-y_0=Pc.
\end{equation}

First suppose \(c=0\). Since \(T\ne\varnothing\),
\eqref{eq:w1-state-table} gives \(T=\{1,q\}\) and \(d=q\). Moreover
\(x_0=y_0\), so
\[
\rho=-\atom{y_1}=-\atom{\Delta-Pq}.
\]
We keep this value for comparison below.

Now suppose \(c\ne0\), so \(c=\pm1\). If \(\alpha=0\), then again
\(\rho=-\atom{y_1}\), so \(y_1\) and
\(d=(\Delta-y_1)/P\) are determined. Here \(d=1\) for \(T=\{q\}\) and
\(d=q-1\) for \(T=\{1\}\), while the case \(c=0\) has \(d=q\). Hence
\(T\) is determined.

Assume \(\alpha\ne0\). By \eqref{eq:w1-top-signature},
\[
y_1\equiv\Delta\pmod P,
\qquad
x_0\equiv y_0\pmod P.
\]
If \(x_0\not\equiv\Delta\pmod P\), then \(x_0,y_0\) lie in one residue
class modulo \(P\), different from the class of \(y_1\). Keeping in \(\rho\)
only the terms \(\atom z\) with \(z\equiv\Delta\pmod P\) therefore gives
\(-\atom{y_1}\). Thus \(y_1\), then \(d=(\Delta-y_1)/P\), and hence \(T\)
are determined.

It remains to consider
\(x_0\equiv y_0\equiv\Delta\pmod P\). Replace every index \(z\) in
\eqref{eq:w1-top-signature} by \((z-\Delta)/P\). For \(T=\{q\}\), set
\(t=(x_0-\Delta)/P\); here \(c=1,d=1\), so
\(y_1=\Delta-P\) and \(y_0=x_0-P\). For \(T=\{1\}\), set
\(\xi=(x_0-\Delta)/P\); here \(c=-1,d=q-1\), so
\(y_1=\Delta-P(q-1)\) and \(y_0=x_0+P\). Together with the case
\(T=\{1,q\}\) above, the three resulting expressions are
\begin{align*}
\mathsf A_t&=-\atom{-1}+\alpha(\atom t-\atom{t-1}) &&(T=\{q\}),\\
\mathsf B_\xi&=-\atom{-(q-1)}+\alpha(\atom\xi-\atom{\xi+1}) &&(T=\{1\}),\\
\mathsf C&=-\atom{-q} &&(T=\{1,q\}).
\end{align*}

For \(\nu=\sum_{j\in\Z}\nu(j)\atom j\), write
\(S_\nu(k)=\sum_{j\le k}\nu(j)\). If \(\ell<r\), then the prefix sum of
\(\atom\ell-\atom r\) equals \(1\) at
\(\ell,\ell+1,\ldots,r-1\) and \(0\) elsewhere. Since \(\alpha\ne0\), the
case \(q=3\) is impossible: it would force \(p=2\) and
\(\alpha=q-1\equiv0\pmod p\). Hence \(q\ge5\).

If \(\mathsf A_t=\mathsf C\), taking prefix sums gives
\[
\mathbf1_{\{-q,-q+1,\ldots,-2\}}
=\alpha\mathbf1_{\{t-1\}},
\]
which is impossible. If \(\mathsf B_\xi=\mathsf C\), prefix sums give
\(\mathbf1_{\{-q\}}+\alpha\mathbf1_{\{\xi\}}=0\), so
\(\xi=-q\) and \(\alpha=-1\). Then \(\alpha\equiv q-1\pmod p\) gives
\(q\equiv0\pmod p\), impossible. Finally, if
\(\mathsf A_t=\mathsf B_\xi\), prefix sums give
\[
\mathbf1_{\{-(q-1),-(q-2),\ldots,-2\}}
=\alpha\bigl(\mathbf1_{\{t-1\}}+\mathbf1_{\{\xi\}}\bigr),
\]
but the left-hand side is nonzero at \(q-2\ge3\) integers and the
right-hand side at at most two. Thus the three possible choices of \(T\)
are distinguishable from the ordinary spectrum.
\end{proof}

\begin{proposition}
\label{prop:w1-immediate}
Suppose \(p\) is odd and \(p\nmid(q-1)\).  The ordinary spectrum of a
connected \(\ICG(p^a q,\mathcal D)\) uniquely determines the spectral counting measure \(\M_H\) of \(H\).
\end{proposition}

\begin{proof}
By Lemma~\ref{lem:w1-top-state}, the ordinary spectrum determines \(T\),
and hence \(c,d\). Since \(\Delta\) is known,
\(y_1=\Delta-Pd\) and \(\Q_d(y_1)\) are known.

Let \(\rho\) denote the coefficientwise reduction of
\(\M_G-\atom\Delta\) modulo \(p\), and let \(\alpha\in\F_p\) be the
residue class of \(q-1\). By hypothesis \(\alpha\neq0\), so it is
invertible in \(\F_p\). From \eqref{eq:w1-top-signature},
\[
\alpha^{-1}\bigl(\rho+\atom{y_1}\bigr)
=
\atom{x_0}-\atom{y_0}.
\]
Suppose first that \(c=\pm1\). Then \(x_0-y_0=Pc\neq0\), so
\(x_0\) and \(y_0\) are distinct integers. Since \(p\) is odd,
\(1\neq-1\) in \(\F_p\). Thus the coefficient \(+1\) identifies \(x_0\), while the coefficient
\(-1\) identifies \(y_0\), so both are recovered from the ordinary spectrum.

If \(c=0\), then \(x_0=y_0\) and \(\Q_c(y_0)=\Q_0(y_0)=0\). In this case
the value of \(y_0\) is irrelevant for recovering \(\M_H\).

Consequently, in every case both correction terms in
\eqref{eq:w1-special-split} are known. Rearranging that exact identity
gives
\begin{equation*}
\M_H
=
\frac1p
\bigl(
\M_G-(q-1)\Q_c(y_0)-\Q_d(y_1)
\bigr).
\end{equation*}
The division is coefficientwise exact because the numerator is
identically \(p\M_H\) in \(\mathcal A\).

\end{proof}

\subsection{Recursion for \texorpdfstring{\(p\mid(q-1)\)}{p dividing q-1}}

Assume \(p\mid(q-1)\), and write
\[
q-1=p^sw,\qquad s=v_p(q-1)\ge1,\qquad p\nmid w.
\]
Hence the term \((q-1)\Q_c(y_0)\) in \eqref{eq:w1-special-split} vanishes modulo \(p\). The recursion below divides an exact
identity by \(p\) one step at a time. If \(s<a\), it reaches the level
\(r=a-s\); if \(s\ge a\), it can reach \(r=0\). The same setup applies to
all primes \(p\). For \(r\ge1\) and
\(y,\tau\in\Z\), define
\begin{equation}\label{eq:w1-Qrc}
\Q_{r,\tau}(y)
=
\atom{y+\tau p^r}+(p-1)\atom y-p\atom{y+\tau p^{r-1}}.
\end{equation}
At the top level \(r=a\), this agrees with the correction
\(\Q_\tau(y)\) defined in \eqref{eq:gen-Q}.

\paragraph{Recursive setup.}
Set \(\mathcal D_a=\mathcal D\). For \(r=a,a-1,\ldots,1\), define
\(\mathcal D_{r-1}\) and \(T_r\) by the unique decomposition
\[
\mathcal D_r=p\mathcal D_{r-1}\mathbin{\dot\cup}T_r,
\qquad T_r\subseteq\{1,q\}.
\]
For \(0\le r\le a\), let \(\M_r\) be the spectral counting measure of
\(\ICG(p^rq,\mathcal D_r)\), and let \(\Delta_r\) be the eigenvalue
corresponding to the trivial character. For \(1\le r\le a\), let \((c_r,d_r)\)
be the pair associated with \(T_r\) by \eqref{eq:w1-state-table}.
The set \(T_r\) may be empty,
so \(d_r\in\{0,1,q-1,q\}\).

For \(1\le r\le a\) and \(v\in\{0,1\}\), let
\(x_{r,v},y_{r,v}\), and \(\gamma_{r,v}\) denote the quantities
\(x_v,y_v\), and \(\gamma_v\) from Proposition~\ref{prop:gen-spectral-vector-split}
applied to the level-\(r\) split
\(\mathcal D_r=p\mathcal D_{r-1}\mathbin{\dot\cup}T_r\). At the top level,
the recursive notation agrees with the original notation:
\[
\mathcal D_{a-1}=\mathcal E,\qquad
\M_a=\M_G,\qquad \M_{a-1}=\M_H,
\]
and \(x_{a,v}=x_v\), \(y_{a,v}=y_v\), and
\(\gamma_{a,v}=\gamma_v\); in particular, \(y_{a,0}=y_0\). Applying
Proposition~\ref{prop:gen-measure-split} to
\(\ICG(p^rq,\mathcal D_r)\) with \(b=1\) gives
\[
\M_r
=
p\M_{r-1}
+(q-1)\Q_{r,c_r}(y_{r,0})
+\Q_{r,d_r}(y_{r,1}).
\]
Since \(x_{r,1}\) and \(\gamma_{r,1}\) are the trivial-character eigenvalues of the level-\(r\) and level-\((r-1)\) graphs, respectively, regularity gives \(x_{r,1}=\Delta_r\) and \(\gamma_{r,1}=\Delta_{r-1}\). Applying
\eqref{eq:gen-top-difference} at level \(r\), and setting
\(\eta_r:=y_{r,1}\), yields
\begin{equation}\label{eq:w1-level-relations}
\eta_r=\Delta_r-p^rd_r,
\qquad
\Delta_{r-1}=\eta_r+p^{r-1}d_r.
\end{equation}
Then \eqref{eq:w1-Qrc} gives
\(\Q_{r,d_r}(\eta_r)
=\atom{\Delta_r}+(p-1)\atom{\eta_r}
-p\atom{\Delta_{r-1}}\). Since \(q-1=p^sw\), substitution yields
\begin{equation}\label{eq:w1-layer-recurrence}
\M_r
=
p\M_{r-1}
+(p-1)\atom{\eta_r}
+\atom{\Delta_r}
-p\atom{\Delta_{r-1}}
+p^sw\Q_{r,c_r}(y_{r,0}).
\end{equation}
Reducing coefficientwise modulo \(p\) gives
\begin{equation}\label{eq:w1-clean-congruence}
\M_r\equiv\atom{\Delta_r}-\atom{\eta_r}\pmod p.
\end{equation}

For \(Z\in\mathcal A\) and \(1\le r\le a-1\), whenever the numerator
below is coefficientwise divisible by \(p\), define
\begin{equation*}
\mathcal R_r(Z)
=
\frac{
Z-(p-1)\atom{\eta_r}-\atom{\Delta_r}
+p\atom{\Delta_{r-1}}
}{p}.
\end{equation*}

By Lemma~\ref{lem:w1-top-state}, \(T_a=T\) is determined; hence
\(c_a=c\) and \(d_a=d\) are determined as well. Since
\(\Delta_a=\Delta\) is known, \eqref{eq:w1-level-relations} with
\(r=a\) determines
\(\eta_a=y_1=\Delta_a-p^ad_a\) and
\(\Delta_{a-1}=\eta_a+p^{a-1}d_a\). Thus both \(\Q_d(y_1)\) and \(\Delta_{a-1}\) are known. Subtracting
\(\Q_d(y_1)\) from \eqref{eq:w1-special-split} and dividing coefficientwise
by \(p\) gives
\begin{equation}\label{eq:w1-K}
\boxed{
K:=\frac{\M_G-\Q_d(y_1)}p
=
\M_H+p^{s-1}w\Q_{a,c}(y_0).
}
\end{equation}
Set \(K_0:=K\). For \(c=0\), the correction term is zero.

\begin{lemma}\label{lem:w1-delayed}
Assume \(s\ge1\), and put
\[
J=\min\{s-1,a-1\}.
\]
Starting from the known pair \((K_0,\Delta_{a-1})\), the recursion determines all pairs
\[
(K_j,\Delta_{a-1-j}),\qquad 0\le j\le J.
\]
More precisely, for \(0\le j<J\), set
\(r=a-1-j\). Then the pair \((K_j,\Delta_r)\) uniquely
determines \(T_r\) and \(\Delta_{r-1}\), and the numerator defining
\(\mathcal R_r(K_j)\) is coefficientwise divisible by \(p\). Hence
\(K_{j+1}:=\mathcal R_r(K_j)\) is well defined.

Throughout the recursion,
\begin{equation}\label{eq:w1-delayed}
K_j
=
\M_{a-1-j}
+
\sum_{i=0}^{j}
p^{\,s-1-(j-i)}w\,
\Q_{a-i,c_{a-i}}(y_{a-i,0}),
\qquad 0\le j\le J.
\end{equation}
\end{lemma}

\begin{proof}
By \eqref{eq:w1-K},
\[
K_0=\M_{a-1}+p^{s-1}w\Q_{a,c_a}(y_{a,0}),
\]
so \eqref{eq:w1-delayed} holds for \(j=0\). Thus the initial pair \((K_0,\Delta_{a-1})\) is known.

Now let \(0\le j<J\), set \(r=a-1-j\), and suppose that
\((K_j,\Delta_r)\) is known and that
\eqref{eq:w1-delayed} holds. Since \(j<J\le s-1\) and
\(j<J\le a-1\), we have \(s-1-j\ge1\) and \(r\ge1\). Thus every
correction term in \eqref{eq:w1-delayed} vanishes coefficientwise
modulo \(p\), and by \eqref{eq:w1-clean-congruence},
\[
K_j\equiv \M_r\equiv \atom{\Delta_r}-\atom{\eta_r}\pmod p.
\]
Now \(\eta_r=\Delta_r-p^rd_r\), where \(d_r\in\{0,1,q-1,q\}\), and the four resulting values of \(\eta_r\) are pairwise distinct. Since \(\Delta_r\) is known, the above congruence therefore uniquely determines \(d_r\): the zero measure corresponds to \(d_r=0\), while otherwise the nonzero coefficient at the index different from \(\Delta_r\) determines \(\eta_r\). Equation~\eqref{eq:w1-state-table} then determines \(T_r\) and \(c_r\), and \(\Delta_{r-1}=\eta_r+p^{r-1}d_r\) determines \(\Delta_{r-1}\). Thus \(\mathcal R_r\) is known.

Substituting the expression for \(\M_r\) from
\eqref{eq:w1-layer-recurrence} into \eqref{eq:w1-delayed}, we obtain
\[
\begin{aligned}
&K_j-(p-1)\atom{\eta_r}-\atom{\Delta_r}+p\atom{\Delta_{r-1}}\\
&\quad=
p\left(
\M_{r-1}
+p^{s-1}w\Q_{r,c_r}(y_{r,0})
+
\sum_{i=0}^{j}
p^{\,s-2-(j-i)}w\,
\Q_{a-i,c_{a-i}}(y_{a-i,0})
\right).
\end{aligned}
\]
Thus the numerator defining \(\mathcal R_r(K_j)\) is coefficientwise divisible by
\(p\), so \(K_{j+1}:=\mathcal R_r(K_j)\) is well defined. Dividing by \(p\),
using \(r=a-(j+1)\), and combining the new correction with the accumulated
ones gives
\[
K_{j+1}
=
\M_{a-2-j}
+
\sum_{i=0}^{j+1}
p^{\,s-1-((j+1)-i)}w\,
\Q_{a-i,c_{a-i}}(y_{a-i,0}),
\]
which is \eqref{eq:w1-delayed} at the next step. Hence the pair \((K_{j+1},\Delta_{r-1})\) provides the input for the next recursion step, and the recursion continues through \(K_J\).
\end{proof}

\subsection{Odd primes with \texorpdfstring{\(p\mid(q-1)\)}{p dividing q-1}}

Proposition~\ref{prop:w1-immediate} handles the odd-prime case \(p\nmid(q-1)\). Assume now that \(p\) is odd and \(p\mid(q-1)\), and write
\[
q-1=p^sw,\qquad s=v_p(q-1)\ge1,\qquad p\nmid w.
\]
Since \(p\) and \(q\) are odd, \(w\) is a positive even integer. We distinguish the two ranges \(s\ge a\) and \(1\le s<a\).

\begin{lemma}\label{lem:w1-large-valuation}
Suppose that \(p\) is odd and \(s=v_p(q-1)\ge a\). Then the ordinary
spectrum uniquely determines \(\mathcal D\).
\end{lemma}

\begin{proof}
Since \(w\) is positive and even, \(q-1=p^sw\ge2p^a\), so \(p^a<q\).
The conclusion follows from \cite[Theorem~1.3(a)]{LiLiu2024}.
\end{proof}

It remains to consider \(1\le s<a\).

\begin{lemma}\label{lem:w1-odd-critical}
Assume that \(p\) is odd and \(1\le s<a\). Put \(m=a-s\ge1\),
\(L=p^s\), and \(W=w\). Then the pair \((K_{s-1},\Delta_m)\) is determined and
\begin{equation}\label{eq:w1-critical-odd}
K_{s-1}
\equiv
\atom{\Delta_m}-\atom{\eta_m}
+
W\bigl(\atom{x_0}-\atom{y_0}\bigr)
\pmod p.
\end{equation}
If \(c=0\), no recovery of \(y_0\) is required. Assume
\(c\in\{-1,1\}\). Since
\(\eta_m\equiv\Delta_m\pmod{p^m}\) and
\(x_0\equiv y_0\pmod{p^m}\), if
\(x_0\not\equiv\Delta_m\pmod{p^m}\), then
\(\Delta_m,\eta_m\) lie in one residue class modulo \(p^m\), while
\(x_0,y_0\) lie in another. Keeping only the terms whose indices are
congruent to \(x_0\pmod{p^m}\) gives
\[
W\bigl(\atom{x_0}-\atom{y_0}\bigr),
\]
which determines \(x_0\) and \(y_0\), since \(p\) is odd and \(p\nmid W\).
Otherwise, \(x_0\equiv y_0\equiv\Delta_m\pmod{p^m}\); write
\(y_0=\Delta_m+p^mt\) and \(d=d_m\). After replacing each index \(z\) by
\((z-\Delta_m)/p^m\) and subtracting the known term \(\atom0\), we obtain
\begin{equation}\label{eq:w1-odd-signature}
\nu_{d,t}
=
-\atom{-d}
+
W\bigl(\atom{t+cL}-\atom t\bigr),
\qquad
d\in\{0,1,LW,LW+1\},
\end{equation}
where the coefficients are understood in \(\F_p\).
\end{lemma}

\begin{proof}
Since \(s<a\), Lemma~\ref{lem:w1-delayed} reaches \(j=s-1\) and determines
\((K_{s-1},\Delta_m)\). In \eqref{eq:w1-delayed}, the correction from level
\(a-i\) then has coefficient \(p^iw\), so the terms with \(i\ge1\) vanish
modulo \(p\). Using \eqref{eq:w1-clean-congruence}, \eqref{eq:w1-Qrc}, and
\eqref{eq:w1-known-top-relations} gives \eqref{eq:w1-critical-odd}. Here
\eqref{eq:w1-level-relations} gives
\(\eta_m=\Delta_m-p^md_m\), where
\(d_m\in\{0,1,q-1,q\}\), and
\eqref{eq:w1-known-top-relations} gives
\(x_0-y_0=cp^a=cp^mL\).

If \(c=0\), the original correction vanishes, so no recovery of \(y_0\) is
required. Assume \(c\in\{-1,1\}\). Since
\(\eta_m\equiv\Delta_m\pmod{p^m}\) and
\(x_0\equiv y_0\pmod{p^m}\), the indices in
\(\atom{\Delta_m}-\atom{\eta_m}\) are congruent to
\(\Delta_m\pmod{p^m}\), while the indices in
\(W(\atom{x_0}-\atom{y_0})\) are congruent to
\(x_0\pmod{p^m}\). If
\(x_0\not\equiv\Delta_m\pmod{p^m}\), these are two different residue
classes. Keeping only the terms in \eqref{eq:w1-critical-odd} whose indices
are congruent to \(x_0\pmod{p^m}\) gives
\[
W\bigl(\atom{x_0}-\atom{y_0}\bigr).
\]
Because \(p\) is odd and \(p\nmid W\), the coefficient \(W\) is nonzero
and different from \(-W\) in \(\F_p\). Hence the coefficient \(W\)
identifies \(x_0\), while the coefficient \(-W\) identifies \(y_0\).
Otherwise \(x_0\equiv y_0\equiv\Delta_m\pmod{p^m}\). Write
\(y_0=\Delta_m+p^mt\). Since \(x_0-y_0=cp^mL\),
\(x_0=\Delta_m+p^m(t+cL)\). Replace every index \(z\) in
\eqref{eq:w1-critical-odd} by \((z-\Delta_m)/p^m\), and subtract the known
term \(\atom0\). Since \(q-1=LW\), we have
\(d_m\in\{0,1,LW,LW+1\}\). Writing \(d:=d_m\), the resulting measure is
\eqref{eq:w1-odd-signature}.
\end{proof}

\begin{lemma}\label{lem:w1-odd-rigidity}
Let \(p\) be odd and \(s\ge1\), put \(L=p^s\), and let \(W\ge2\) be
even with \(p\nmid W\). Fix \(c\in\{-1,1\}\). Then the map defined by
\eqref{eq:w1-odd-signature},
\[
(d,t)\longmapsto\nu_{d,t},
\qquad d\in\{0,1,LW,LW+1\},\quad t\in\Z,
\]
is injective.
\end{lemma}

\begin{proof}
Suppose
\(\nu_{d_1,t_1}=\nu_{d_2,t_2}\). By
\eqref{eq:w1-odd-signature},
\[
-\atom{-d_1}
+W\bigl(\atom{t_1+cL}-\atom{t_1}\bigr)
=
-\atom{-d_2}
+W\bigl(\atom{t_2+cL}-\atom{t_2}\bigr).
\]
For \(\theta\in\Z\), let \(F_\theta:\Z\to\F_p\) be the
step function defined by \(F_\theta(x)=\mathbf1_{\{x\ge \theta\}}\). For
\(i\in\{1,2\}\), let \(I_i\) be the length-\(L\) integer interval
determined by \(t_i\): if \(c=1\), then
\(I_i=\{t_i,\ldots,t_i+L-1\}\), while if \(c=-1\), then
\(I_i=\{t_i-L,\ldots,t_i-1\}\). Taking prefix sums gives
\begin{equation}\label{eq:w1-prefix-rigidity}
F_{-d_2}-F_{-d_1}
=
cW\bigl(\mathbf1_{I_1}-\mathbf1_{I_2}\bigr).
\end{equation}

Assume first that \(d_1\ne d_2\). Since \(F_{-d_2}\) and
\(F_{-d_1}\) are step functions, their difference is equal to
\(1\) or \(-1\) precisely on the integer interval between their two
jump points and vanishes elsewhere. Its support therefore has length
\(|d_1-d_2|\). Since
\(d_1,d_2\in\{0,1,LW,LW+1\}\), the possible nonzero lengths are
\(1,LW-1,LW,LW+1\).

Let \(k\) be the distance between the starting points of \(I_1\) and
\(I_2\). If \(k<L\), the intervals overlap in \(L-k\) positions and
each contributes \(k\) positions to their symmetric difference; if
\(k\ge L\), they are disjoint. Thus
\[
\bigl|\supp
(\mathbf 1_{I_1}-\mathbf 1_{I_2})\bigr|
=
2\min(k,L).
\]
Since \(L\) is odd and \(W\) is even, the three possible lengths
\(1,LW-1,LW+1\) are odd, whereas the support on the right-hand side
of \eqref{eq:w1-prefix-rigidity} has even size. Hence these three
possibilities are excluded.

It remains to consider the length \(LW\). If \(W>2\), then
\(LW>2L\), while the right-hand side of
\eqref{eq:w1-prefix-rigidity} has support of size at most \(2L\), a
contradiction. Thus only \(W=2\) remains, in which case \(LW=2L\).
Since the left-hand side has contiguous support of length \(2L\),
the two length-\(L\) intervals \(I_1\) and \(I_2\) must be disjoint
and adjacent. On these two adjacent intervals, however, the
right-hand side has coefficients \(cW\) and \(-cW\), whereas the
left-hand side has one constant coefficient throughout its support.
Equality would therefore force \(cW=-cW\), equivalently
\(2cW=0\) in \(\F_p\), impossible because \(p\) is odd and
\(p\nmid W\). Consequently \(d_1=d_2\).

With \(d_1=d_2\), the left-hand side of
\eqref{eq:w1-prefix-rigidity} vanishes, so
\(cW(\mathbf 1_{I_1}-\mathbf 1_{I_2})=0\). Since
\(cW\ne0\) in \(\F_p\), we obtain \(I_1=I_2\). Since \(c\) is fixed, equality \(I_1=I_2\) implies equality of their starting points, and hence \(t_1=t_2\). Therefore
\((d,t)\mapsto\nu_{d,t}\) is injective.
\end{proof}

\begin{proposition}\label{prop:w1-odd-delayed}
Assume that \(p\) is odd and \(1\le s=v_p(q-1)<a\). Then the ordinary spectrum of a connected graph \(\ICG(p^aq,\mathcal D)\) uniquely determines \(\M_H\).
\end{proposition}

\begin{proof}
If \(c=0\), then by \eqref{eq:w1-Qrc},
\[
\Q_{a,c}(y_0)
=\Q_{a,0}(y_0)
=\atom{y_0}+(p-1)\atom{y_0}-p\atom{y_0}
=0,
\]
so \eqref{eq:w1-K} gives \(\M_H=K\). Assume
\(c=\pm1\), and put \(m=a-s\). By
Lemma~\ref{lem:w1-odd-critical}, if
\(x_0\not\equiv\Delta_m\pmod{p^m}\), then \(x_0\) and \(y_0\) are recovered
directly. If \(x_0\equiv y_0\equiv\Delta_m\pmod{p^m}\),
Lemma~\ref{lem:w1-odd-rigidity} recovers \(t\), and hence
\(y_0=\Delta_m+p^mt\).

Thus \(y_0\) is determined in all cases. All terms in \eqref{eq:w1-K}
except \(\M_H\) are now known, so
\[
\M_H
=
K-p^{s-1}w\Q_{a,c}(y_0).
\]
\end{proof}

\subsection{The case \texorpdfstring{\(p=2\)}{p=2}}

Now let \(p=2\). Since \(q\) is odd, \(2\mid(q-1)\)
automatically. Write \(q-1=2^sw\), where \(s=v_2(q-1)\ge1\) and \(w\)
is odd. Our goal in this subsection is to recover the spectral counting measure \(\M_H\) of \(H\) from the ordinary
spectrum. We distinguish the two
ranges \(s<a\) and \(s\ge a\). Lemma~\ref{lem:w1-top-state} still
determines \(T\).

When \(p=2\), \eqref{eq:w1-clean-congruence} becomes
\begin{equation*}
\M_r
\equiv
\atom{\Delta_r}+\atom{\eta_r}
\pmod2.
\end{equation*}
At each step of Lemma~\ref{lem:w1-delayed}, \(\Delta_r\) is known: the zero measure gives \(d_r=0\), while otherwise the index different from
\(\Delta_r\) with nonzero coefficient determines \(\eta_r\). Hence the recursion in Lemma~\ref{lem:w1-delayed} remains valid when \(p=2\).

\begin{lemma}\label{lem:w1-binary-critical}
Assume that \(s<a\), and put \(m=a-s\), \(L=2^s\), and \(W=w\). Then
the pair \((K_{s-1},\Delta_m)\) is determined and
\begin{equation}\label{eq:w1-critical-binary}
K_{s-1}
\equiv
\atom{\Delta_m}+\atom{\eta_m}
+
\bigl(\atom{x_0}-\atom{y_0}\bigr)
\pmod2.
\end{equation}
If \(c=0\), no recovery of \(y_0\) is required. Assume
\(c\in\{-1,1\}\). If \(x_0\not\equiv\Delta_m\pmod{2^m}\), then
\eqref{eq:w1-critical-binary} determines the unordered pair \(\{x_0,y_0\}\),
and \(x_0-y_0=c2^a\) determines its ordering. Otherwise,
\(x_0\equiv y_0\equiv\Delta_m\pmod{2^m}\); write
\(y_0=\Delta_m+2^mt\) and \(d=d_m\). After replacing each index \(z\) by
\((z-\Delta_m)/2^m\) and subtracting the known term \(\atom0\), we obtain
\begin{equation}\label{eq:w1-binary-signature}
\sigma_{d,t}
=
-\atom{-d}
+\atom{t+cL}
-\atom t,
\qquad
d\in\{0,1,LW,LW+1\},
\end{equation}
where the coefficients are understood in \(\F_2\).
\end{lemma}

\begin{proof}
Since \(s<a\), Lemma~\ref{lem:w1-delayed}, with \(p=2\), reaches
\(j=s-1\) and determines \((K_{s-1},\Delta_m)\). Setting \(j=s-1\) in
\eqref{eq:w1-delayed} and reducing modulo \(2\) eliminates all terms with
\(i\ge1\). Since \(W=w\equiv1\pmod2\), equations
\eqref{eq:w1-clean-congruence}, \eqref{eq:w1-Qrc}, and
\eqref{eq:w1-known-top-relations} give \eqref{eq:w1-critical-binary}. By \eqref{eq:w1-level-relations},
\(\eta_m=\Delta_m-2^md_m\), where
\(d_m\in\{0,1,q-1,q\}\); moreover,
\eqref{eq:w1-known-top-relations} gives
\(x_0-y_0=c2^a=c2^mL\).

If \(c=0\), no recovery of \(y_0\) is required. Assume \(c=\pm1\).
Since \(\eta_m\equiv\Delta_m\pmod{2^m}\) and
\(x_0\equiv y_0\pmod{2^m}\), the first two terms in
\eqref{eq:w1-critical-binary} are indexed by integers congruent to
\(\Delta_m\pmod{2^m}\), while the last two are indexed by integers congruent
to \(x_0\pmod{2^m}\). If a nonzero coefficient occurs at an index not
congruent to \(\Delta_m\), then the two such indices are \(x_0\) and \(y_0\),
so \(\{x_0,y_0\}\) is determined; the relation \(x_0-y_0=c2^a\) determines
their ordering.

Otherwise \(x_0\equiv y_0\equiv\Delta_m\pmod{2^m}\). Write
\(y_0=\Delta_m+2^mt\). By \eqref{eq:w1-known-top-relations},
\(x_0=\Delta_m+2^m(t+cL)\). Replace every index \(z\) in
\eqref{eq:w1-critical-binary} by \((z-\Delta_m)/2^m\), subtract the known
term \(\atom0\), and write \(d:=d_m\). The resulting measure is
\eqref{eq:w1-binary-signature}.
\end{proof}

\begin{lemma}\label{lem:w1-binary-rigidity}
Let \(s\ge1\), put \(L=2^s\), and let \(W\) be a positive odd integer. Fix
\(c\in\{-1,1\}\). Then the map
\[
(d,t)\longmapsto \sigma_{d,t},
\qquad d\in\{0,1,LW,LW+1\},\quad t\in\Z,
\]
defined by \eqref{eq:w1-binary-signature} is injective.
\end{lemma}

\begin{proof}
Suppose that \(\sigma_{d_1,t_1}=\sigma_{d_2,t_2}\). By \eqref{eq:w1-binary-signature},
\[
-\atom{-d_1}+\atom{t_1+cL}-\atom{t_1}
=
-\atom{-d_2}+\atom{t_2+cL}-\atom{t_2}.
\]
Use the definitions of \(F_\theta\) and \(I_i\) from the proof of Lemma~\ref{lem:w1-odd-rigidity}, with \(F_\theta\) now taking values in \(\F_2\). Taking prefix sums gives
\begin{equation}\label{eq:w1-binary-prefix}
F_{-d_2}-F_{-d_1}
=
c\bigl(\mathbf 1_{I_1}-\mathbf 1_{I_2}\bigr).
\end{equation}
Since addition and subtraction coincide in \(\F_2\), and \(c\in\{-1,1\}\), the right-hand side is \(\mathbf 1_{I_1\triangle I_2}\).

Suppose first that \(d_1\ne d_2\). The left-hand side of \eqref{eq:w1-binary-prefix} is supported on the integer interval between the two jump points and therefore has support size \(|d_1-d_2|\). Since \(d_1,d_2\in\{0,1,LW,LW+1\}\), the possible nonzero values are \(1\), \(LW-1\), \(LW\), and \(LW+1\).

On the other hand, \eqref{eq:w1-binary-prefix} implies that \(I_1\triangle I_2\) is a single integer interval. For two distinct integer intervals of the same length \(L\), this is possible only when they are disjoint and adjacent, in which case \(|I_1\triangle I_2|=2L\). Hence \(|d_1-d_2|=2L\).

Since \(L=2^s\) is even and \(W\) is odd, the values \(1\), \(LW-1\), and \(LW+1\) are odd and therefore cannot equal \(2L\). The remaining possibility \(LW=2L\) would imply \(W=2\), contradicting that \(W\) is odd. Thus \(d_1=d_2\).

With \(d_1=d_2\), \eqref{eq:w1-binary-prefix} gives \(\mathbf 1_{I_1\triangle I_2}=0\), so \(I_1=I_2\). Since \(c\) is fixed, equality of these intervals implies \(t_1=t_2\). Hence \((d_1,t_1)=(d_2,t_2)\), proving injectivity.
\end{proof}

\begin{lemma}\label{lem:w1-binary-bottom-recovery}
Assume that \(s\ge a\) and \(c\in\{-1,1\}\). Then the ordinary spectrum determines \(y_0\).
\end{lemma}

\begin{proof}
Since \(s\ge a\), Lemma~\ref{lem:w1-delayed} can be applied through all
\(a-1\) steps and reaches \(r=0\). Thus \((K_{a-1},\Delta_0)\) is known and
\begin{equation}\label{eq:w1-binary-bottom}
K_{a-1}
=
\M_0
+
\sum_{i=0}^{a-1}
2^{\,s-a+i}w
\Q_{a-i,c_{a-i}}(y_{a-i,0}).
\end{equation}

Since \(\mathcal D_0\subseteq\Div^*(q)=\{1\}\), the level-zero graph
is either edgeless or complete. Hence
\[
\M_0=q\atom0,\quad \Delta_0=0,
\qquad\text{or}\qquad
\M_0=\atom{q-1}+(q-1)\atom{-1},\quad \Delta_0=q-1.
\]
Thus the known value \(\Delta_0\) determines \(\M_0\).

Suppose first that \(s=a\). Since \(q\) is odd, in either case
\(\M_0\equiv\atom{\Delta_0}\pmod2\). Subtracting \(\atom{\Delta_0}\) from \eqref{eq:w1-binary-bottom} and
reducing modulo \(2\) leaves only the \(i=0\) term. Since
\(w\equiv1\pmod2\), we obtain
\begin{equation*}
K_{a-1}-\atom{\Delta_0}
\equiv
\Q_{a,c}(y_0)
\equiv
\atom{x_0}+\atom{y_0}
\pmod2.
\end{equation*}

Now suppose that \(s>a\). By \eqref{eq:w1-binary-bottom}, every
coefficient of \(K_{a-1}-\M_0\) is divisible by \(2^{s-a}\). Hence the
coefficientwise quotient is well defined and equals
\[
\frac{K_{a-1}-\M_0}{2^{s-a}}
=
\sum_{i=0}^{a-1}
2^iw
\Q_{a-i,c_{a-i}}(y_{a-i,0}).
\]
Reducing modulo \(2\) again leaves only the \(i=0\) term, and hence
\begin{equation*}
\frac{K_{a-1}-\M_0}{2^{s-a}}
\equiv
\Q_{a,c}(y_0)
\equiv
\atom{x_0}+\atom{y_0}
\pmod2.
\end{equation*}

Thus in either case the reduction determines the unordered pair
\(\{x_0,y_0\}\). Since \(T\) has already determined \(c\)
and \(x_0-y_0=c2^a\), the ordering of the pair is uniquely determined and \(y_0\) is
recovered.
\end{proof}

\begin{proposition}\label{prop:w1-binary}
For every odd prime \(q\) and every \(a\ge1\), the ordinary spectrum of a connected \(\ICG(2^aq,\mathcal D)\) uniquely determines \(\M_H\).
\end{proposition}

\begin{proof}
If \(c=0\), then \(\Q_{a,c}(y_0)=0\) by \eqref{eq:w1-Qrc}, so \eqref{eq:w1-K} gives \(\M_H=K\). Assume \(c=\pm1\).

If \(s<a\), put \(m=a-s\). Lemma~\ref{lem:w1-binary-critical} recovers \(x_0,y_0\) directly when
\(x_0\not\equiv\Delta_m\pmod{2^m}\). When
\(x_0\equiv y_0\equiv\Delta_m\pmod{2^m}\),
Lemma~\ref{lem:w1-binary-rigidity} determines \(t\), and hence
\(y_0=\Delta_m+2^mt\). If \(s\ge a\),
Lemma~\ref{lem:w1-binary-bottom-recovery} determines \(y_0\). Thus
\(y_0\) is determined in all cases.

Returning to \eqref{eq:w1-K}, we obtain
\[
\M_H=K-2^{s-1}w\Q_{a,c}(y_0).
\]
\end{proof}

\begin{theorem}\label{thm:w1-recovery}
Let \(G=\ICG(p^aq,\mathcal D)\), where \(p<q\) are primes and \(a\ge1\). If \(G\) is connected, then
its ordinary spectrum determines \(T\) and the exact spectral counting measure \(\M_H\) of \(H\). No connectivity assumption is imposed on \(H\), and its divisor set may be empty.
\end{theorem}

\begin{proof}
Lemma~\ref{lem:w1-top-state} determines \(T\).

Suppose first that \(p\) is odd. If \(p\nmid(q-1)\),
Proposition~\ref{prop:w1-immediate} determines \(\M_H\). If
\(p\mid(q-1)\) and \(v_p(q-1)\ge a\),
Lemma~\ref{lem:w1-large-valuation} gives the stronger direct
reconstruction of \(\mathcal D\). Since \(T\) is already known, this
also determines \(\mathcal E\), and hence \(\M_H\). If
\(1\le v_p(q-1)<a\),
Proposition~\ref{prop:w1-odd-delayed} determines \(\M_H\).

For \(p=2\), Proposition~\ref{prop:w1-binary} determines \(\M_H\)
for every \(a\ge1\). These cases exhaust all possibilities.

\end{proof}

\section{Induction and proof of the main theorem}

Theorem~\ref{thm:w1-recovery} requires only \(G\) to be connected; \(H\)
may be disconnected, and its divisor set may be empty. Accordingly, the induction below is formulated for all
divisor sets.

\begin{proof}[Proof of Theorem~\ref{thm:main}]
Fix primes \(p<q\). For each integer \(r\ge0\), let \(\mathcal P(r)\) denote
the statement that, for every divisor set
\(\mathcal D'\subseteq\Div^*(p^rq)\), the ordinary
spectrum of \(\ICG(p^rq,\mathcal D')\) uniquely determines
\(\mathcal D'\). We prove \(\mathcal P(a)\) by strong induction on \(a\).

The base case \(\mathcal P(0)\) concerns integral circulant graphs of order \(q\), and follows
from Lemma~\ref{lem:gen-prime-power}. Now let \(a\ge1\), and assume
that \(\mathcal P(r)\) holds for every \(0\le r<a\). Thus the induction hypothesis applies to every integral circulant graph of order \(p^rq\)
with \(r<a\), regardless of connectivity or whether its divisor set is
empty.

Let \(G=\ICG(p^aq,\mathcal D)\), where
\(\mathcal D\subseteq\Div^*(p^aq)\), and suppose that
its ordinary spectrum is known. If the entire spectrum is zero, then the real symmetric
adjacency matrix of \(G\) has only the eigenvalue \(0\), and hence is
the zero matrix. Therefore \(\mathcal D=\varnothing\).

Assume now that \(\mathcal D\neq\varnothing\), and put
\(g=\gcd(\mathcal D)=p^iq^j\), where
\(0\le i\le a\) and \(j\in\{0,1\}\).
If \(g>1\), Lemma~\ref{lem:gen-components} shows that \(G\) is the
disjoint union of \(g\) isomorphic connected components, each
isomorphic to
\[
G_0=\ICG(p^{a-i}q^{1-j},\mathcal D/g).
\]
The common degree is a simple eigenvalue of each connected component, so its
multiplicity in \(G\) is exactly \(g\). Hence the ordinary spectrum
determines \(g\), and division of every eigenvalue multiplicity by
\(g\) recovers the ordinary spectrum of \(G_0\).

If \(j=0\), then \(i\ge1\) and \(G_0\) has order
\(p^{a-i}q\), with strictly smaller \(p\)-exponent; hence the strong
induction hypothesis \(\mathcal P(a-i)\) recovers \(\mathcal D/g\). If
\(j=1\), then \(i<a\), since every element of \(\mathcal D\) is a
proper divisor of \(p^aq\). Hence \(a-i\ge1\), and \(G_0\) has
prime-power order \(p^{a-i}\). Lemma~\ref{lem:gen-prime-power}
therefore recovers \(\mathcal D/g\).
In either case, multiplication by the known integer \(g\) recovers
\(\mathcal D\).

It remains to consider \(g=1\). By
Lemma~\ref{lem:gen-components}, \(G\) is then connected. Write
\[
\mathcal D=p\mathcal E\mathbin{\dot\cup}T,
\qquad
H=\ICG(p^{a-1}q,\mathcal E).
\]
By Theorem~\ref{thm:w1-recovery}, the ordinary spectrum of \(G\)
determines \(T\) and the exact spectral counting measure \(\M_H\) of \(H\).
By definition, \(\M_H\) records the numerical eigenvalues of \(H\) with
their multiplicities and thus determines its ordinary spectrum.

If \(a=1\), then \(H\) has prime order \(q\), so Lemma~\ref{lem:gen-prime-power} recovers \(\mathcal E\). If \(a\ge2\), then \(H\) has order \(p^{a-1}q\), so the strong induction hypothesis \(\mathcal P(a-1)\) recovers the divisor set \(\mathcal E\). Thus \(\mathcal E\) is uniquely determined in all cases, and
finally
\[
\mathcal D=p\mathcal E\mathbin{\dot\cup}T.
\]
Hence the ordinary spectrum uniquely determines \(\mathcal D\).
\end{proof}

\section*{Acknowledgements}
This work was supported by the National Natural Science Foundation of China under Grant No.~12071351.


\begin{thebibliography}{99}

\bibitem{AltinisikAydin2024}
E. Alt{\i}n{\i}{\c{s}}{\i}k and S. B. Ayd{\i}n,
Integral circulant graphs and So's conjecture,
\emph{Turkish J. Math. Comput. Sci.} \textbf{16} (2024), no.~1, 169--176.

\bibitem{Apostol1976}
T. M. Apostol,
\emph{Introduction to Analytic Number Theory},
Undergraduate Texts in Mathematics,
Springer-Verlag, New York, 1976.

\bibitem{Cusanza2006}
C. F. Cusanza,
\emph{Integral circulant graphs with the spectral {\'A}d{\'a}m property},
Master's thesis, San Jos\'e State University, 2006.

\bibitem{KlinKovacs2012}
M. Klin and I. Kov\'acs,
Automorphism groups of rational circulant graphs,
\emph{Electron. J. Combin.} \textbf{19} (2012), no.~1, Paper~35.

\bibitem{LiLiu2024}
H. Li and X. Liu,
So's conjecture for integral circulant graphs of 4 types,
\emph{Discrete Math.} \textbf{347} (2024), no.~7, 114044.

\bibitem{MoniusSo2023}
K. M\"onius and W. So,
How many non-isospectral integral circulant graphs are there?,
\emph{Australas. J. Combin.} \textbf{86} (2023), no.~2, 320--335.

\bibitem{Sander2018}
J. W. Sander,
Structural properties and formulae of the spectra of integral circulant graphs,
\emph{Acta Arith.} \textbf{184} (2018), no.~3, 297--315.

\bibitem{SanderSander2015}
J. W. Sander and T. Sander,
On So's conjecture for integral circulant graphs,
\emph{Appl. Anal. Discrete Math.} \textbf{9} (2015), no.~1, 59--72.

\bibitem{SchlagePuchta2021}
J.-C. Schlage-Puchta,
A determinant involving Ramanujan sums and So's conjecture,
\emph{Arch. Math.} \textbf{117} (2021), no.~4, 379--384.

\bibitem{So2006}
W. So,
Integral circulant graphs,
\emph{Discrete Math.} \textbf{306} (2006), no.~1, 153--158.

\bibitem{Zhang2023}
Y. X. Zhang,
On isospectral integral circulant graphs,
arXiv:2310.11545, 2023.


\end{thebibliography}
\end{document}